\documentclass{article}
\usepackage[utf8]{inputenc}
\usepackage{amssymb}
\usepackage{amsthm}
\usepackage{amsmath}
\usepackage{amsfonts}
\usepackage{url}
\usepackage{hyperref}
\usepackage{graphicx}
\newtheorem{thm}{Theorem}[section]

\newtheorem{prop}[thm]{Proposition}
\newtheorem{lem}[thm]{Lemma}
\newtheorem{conj}[thm]{Conjecture}

\theoremstyle{definition}
\newtheorem{defn}[thm]{Definition}

\newtheorem{exmp}[thm]{Example}

\theoremstyle{remark}
\newtheorem{rem}[thm]{Remark}

\title{On an existence problem of periodic points in intervals whose images cover themselves}
\author{Wang Yihan\\2001110016@pku.edu.cn}
\date{April 2022}

\begin{document}

\maketitle

\begin{abstract}
We study a system of intervals $I_1,\ldots,I_k$ on the real line whose images under a continuous map $f$ contain themselves. It's conjectured that there exists a periodic point of period $\le k$ in $I_1\cup \ldots \cup I_k$. We prove the conjecture for $k=5$ in this paper. We also propose a discretization method in attempt to solve the problem.
\end{abstract}

\section{Introduction}

We consider a continuous mapping $f$: $\mathbb{R} \to \mathbb{R}$, and $k$ closed intervals $I_1,I_2,...,I_k$. Throughout this paper, $f^l$ denotes the $l$-th iteration of $f$. 

\begin{defn}
We call $(I_1,I_2,\ldots ,I_k;f)$ a $covering\ system$ if $f(I_1 \cup I_2 \cup \ldots \cup I_k)\supseteq I_1 \cup I_2 \cup \ldots \cup I_k$.
\end{defn}

This concept was first introduced by S. A. Bogatyi and E. T. Shavgulidze in their paper \cite{1988Periodic}, when they tried to obtain an analogue of Sharkovskii's theorem \cite{A1995COEXISTENCE} for an arbitrary tree. They proved the existence of a periodic point of period $\le g(n)$ in a covering system and established an upper bound for $g(n)$ $(g(n)\le 2(n^2-1)g(n-1)^{n-2})$. Moreover, they proposed the following conjecture in the remark. 

\begin{conj}
For any covering system $(I_1,I_2,\ldots,I_k;f)$, there exists a point $x_0\in I_1 \cup I_2 \cup \ldots \cup I_k$ such that $f^l(x_0)=x_0$ for some $l\le k$.
\end{conj}

It is proved in \cite{1988Periodic} that Conjecture 1.2 is true for $k\le 4$. By similar but more detailed discussions, we can extend the results to $k=5$. And we obtain the first main theorem of this paper.

\begin{thm}
For any covering system $(I_1,\ldots,I_5;f)$, there exists a periodic point of period $\le 5$ in $I_1 \cup \ldots \cup I_5$.
\end{thm}

Inspired by the process of proof for $k=5$, we develop new methods which relate the initial problem to a discrete one. Indeed, we propose the following conjecture and show that Conjecture 1.2 is a corollary of it. 

\begin{conj}
For any mapping $f$: $\{1,2,\ldots,n\}$ $\to$ subsets of $\{1,2,\ldots,n\}$ such that $\bigcup_{i=1}^n f(i)=\{1,2,\ldots,n\}$, and for any partition 
$$I_1=\{1,2,\ldots,i_1\},\quad I_2=\{i_1+1,\ldots,i_2\},\quad \ldots,\quad I_k=\{i_{k-1}+1,\ldots,n\},$$
there exists $j\in \{1,2,\ldots,k\}$ and $r,s\in I_j$ such that
$$(convf)^l(\{r,s\})\supseteq \{r,s\},\quad for\ some \quad l\le k.$$
Here, $convf(A)=$ convex hull of $f(A)=\{\min f(A),\ldots,\max f(A)\}$ for any finite set $A\subseteq \mathbb{N}$.
\end{conj}

\begin{thm}
Conjecture 1.4 implies Conjecture 1.2.
\end{thm}

Moreover, we can prove that it suffices to consider the case $f$ is a transitive permutation in the symmetry group $S_n$ in Conjecture 1.4. And we will see Conjecture 1.4 is equivalent to the following one, which concerns only the property of a permutation. 

\begin{defn}
Let $f$ be a transitive permutation in the symmetry group $S_n$, $i.e.$, $f$ can be written as $(i_1i_2\ldots i_n)$, and $A_i=\{i,i+1\}$ $(i=1,2,\ldots,n-1)$ be $(n-1)$ particular sets. We define the $characteristic\ number$ $m_i$ of each $A_i$ as
$$m_i=\min \{m\arrowvert\ (convf)^m(A_i)\supseteq A_i\}.$$
Here, $convf(A)=$ convex hull of $f(A)=\{\min f(A),\ldots,\max f(A)\}$ for any finite set $A\subseteq \mathbb{N}$.\par
And we define the $characteristic\ sequence$ of $f$ to be 
$$m_1'\le m_2'\le \ldots\le m_{n-1}',$$ 
where $\{m_i'\}$ is a rearrangement of $\{m_i\}$.
\end{defn}

\begin{conj}
For any transitive $f$ in the symmetry group $S_n$, the characteristic sequence $\{m_i'\}$ of $f$ satisfies $m_i'\le i$, \ $i=1,2,\ldots,n-1$.
\end{conj}

One of the advantages of Conjecture 1.7 is the independence of the partition. Furthermore, using the directed graph associated to a periodic orbit (see \cite{onedimdynamics1992} page 7 for definitions), we can finally prove:

\begin{thm}
Conjecture 1.7 is equivalent to Conjecture 1.4.
\end{thm}

We organize the paper in the following way.

In section 2 we explain our main idea of proof for $k=5$ through specific examples. In fact, we claim that for $k$ small (at least for $k\le 5$), the existence of a periodic point only depends on the ordering of the end points of the intervals and their images under $f^i(i\le k)$. Therefore, Theorem 1.3 can be proved case by case. 

In section 3 we give new ideas for attempts to prove Conjecture 1.2. We explain in detail how we can reduce the initial problem to a discrete one. And the main results, Theorem 1.5 and Theorem 1.8, are proved in this section.

\section{The case $k=5$}

The purpose of this section is to prove Theorem 1.3. Before coming to the proof for $k=5$, we first take a look at some basic examples which help us understand the problem. We will always assume the mappings under consideration are continuous from now on.

\begin{exmp}[the case $k=1$]
Let $f:I \to \mathbb{R}$ satisfy $f(I) \supseteq I$. Then there exists a fixed point of $f$ in $I$.
\end{exmp}

\begin{rem}
This fact is an easy consequence of the intermediate value theorem in calculus. And we will see later that for $k$ larger, most cases can be reduced to this fundamental one.   
\end{rem}

\begin{exmp}
(See Figure 1) In this example, each interval $I_i$ is mapped to some $f(I_{i_j}) \supseteq I_j$. Then $f$ induces a permutation on ${I_1,\ldots,I_k}$. Thus, $\exists$ $l \le k$, $s.t.$ $f^l(I_1) \supseteq I_1$. And by example 2.1, we can find $x_0 \in I_1$ with $f^l(x_0)=x_0$.This example explains why we think the period is not bigger than $k$ in conjecture 2.1.
\begin{figure}[hb]
    \centering
    \includegraphics[width=\textwidth]{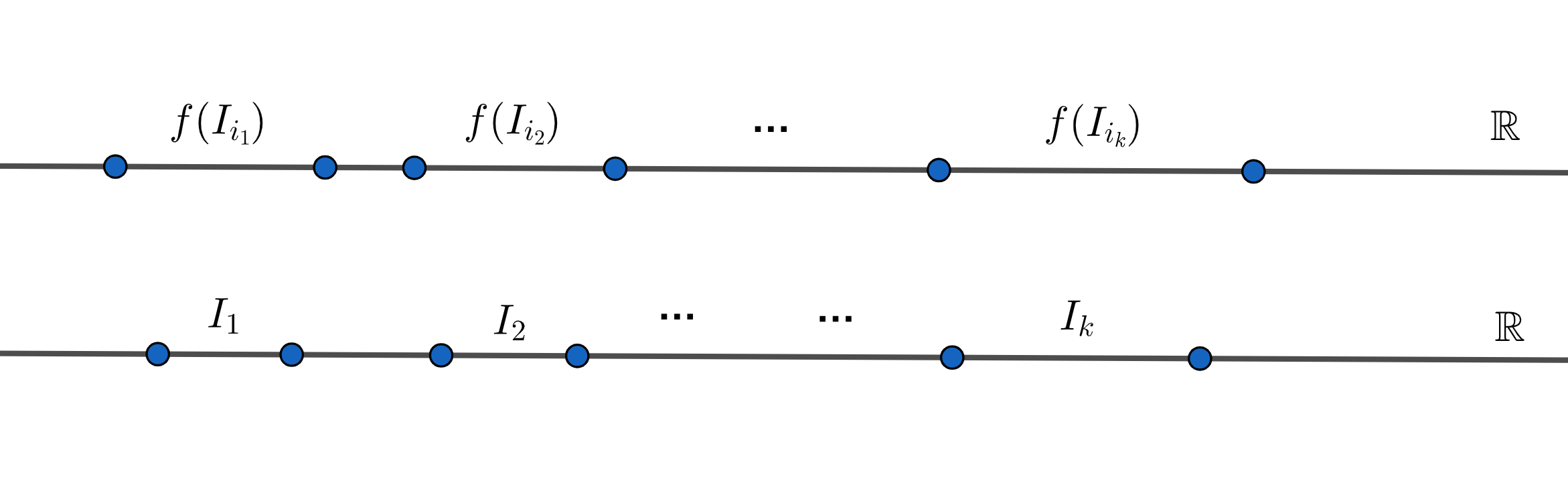}
    \caption{Example 2.3}
    \label{fig:my_label}
\end{figure}
\end{exmp}

We state a well-known lemma which is useful when we want to simplify the initial problem.

\begin{lem}
If $f([a,b]) \supseteq [c,d]$, then there is a subinterval $[a',b'] \subseteq [a,b]$ such that $f([a',b'])=[c,d]$. And furthermore, the end points of $[a',b']$ is mapped to the end points of $[c,d]$.
\end{lem}
\begin{proof}
There are points $x_c$ and $x_d$ in $[a,b]$ such that $f(x_c)=c$ and $f(x_d)=d$. Without loss of generality, we can assume $x_c<x_d$. Then the proof is completed by choosing $a'=\sup f^{-1}(c)\cap [x_c,x_d]$ and $b'=\inf f^{-1}(d)\cap [a',x_d]$.
\end{proof}

\begin{rem}
We see from this lemma that if necessary, we can replace $I_i$ in a covering system $(I_1,\ldots,I_k;f)$ by some subintervals $I_i'$ such that $(I_1',\ldots,I_k';f)$ is still a covering system, but the end points of each $I_i'$ are mapped to the end points of $f(I_i')$.

Moreover, we can assume the initial covering system $(I_1,\ldots,I_k;f)$ is $minimal$, that is, there doesn't exist any covering system $(J_1,\ldots,J_l;f)$ with the same mapping $f$ such that $l\le k$ and each $J_j$ is contained in some $I_i$. (See lemma 1 in \cite{1988Periodic} for more details.) In particular, we can assume that the interior sets of $f(I_i)$ and $f(I_j)$ are disjoint. Of course, we can also assume $I_i$ and $I_j$ are disjoint in Conjecture 1.2, since if not, we can consider a new interval $I_i\cup I_j$ instead.
\end{rem}

Next, we introduce our methods to deal with different cases for $k=5$ through the following three examples. All the covering systems below are assumed to be minimal. Hence, each end point of $I_i$ is mapped into $I_1\cup \ldots \cup I_k$. (See \cite{1988Periodic} lemma 4.)

\begin{figure}[ht]
    \centering
    \includegraphics[width=\textwidth]{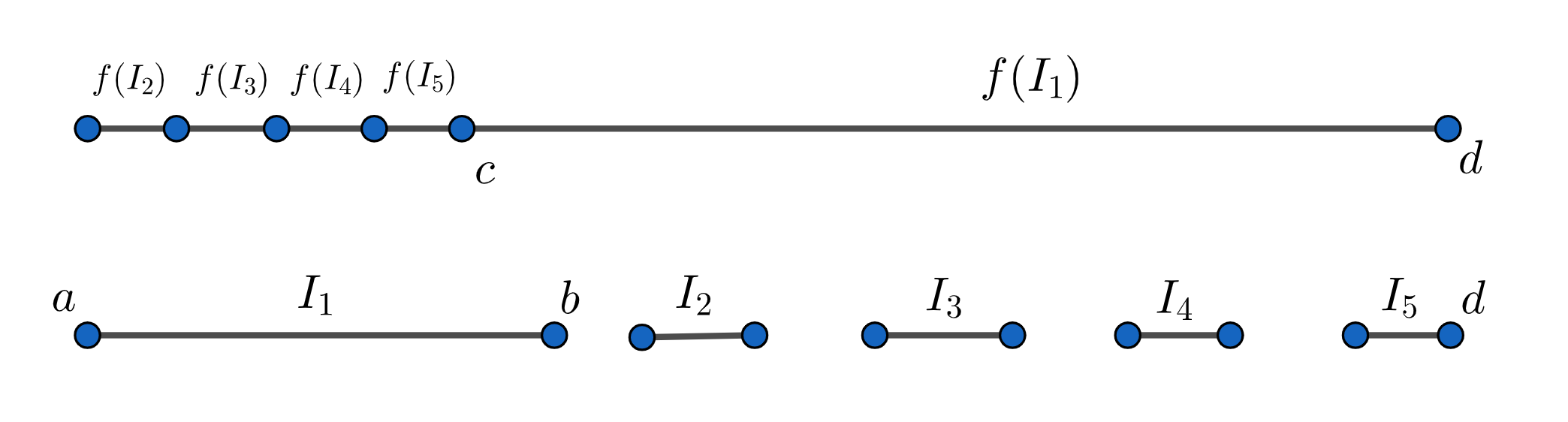}
    \caption{Example 2.6}
    \label{fig:my_label}
\end{figure}

\begin{exmp}
The behavior of $f$ is illustrated in Figure 2.\par
 As mentioned above, we assume the end points of $I_1$ are mapped to the end points of $f(I_1)$. We can assume further that $f(a)=c,\ f(b)=d$. Otherwise, there will exist a fixed point of $f$ in $I_1$. For the same reason, $f(x)>x, \forall x\in I_1$. So we obtain
$$f^2(a)=f(c)>c>a,\quad f^2(b)=f(d)<b,$$
which shows that $f^2(x_0)=x_0$ for some $x_0\in I_1$.\par
We find that $I_1\cap f(I_1)\neq \emptyset$, which is much special in this example.
\end{exmp}

\begin{figure}[ht]
    \centering
    \includegraphics[width=\textwidth]{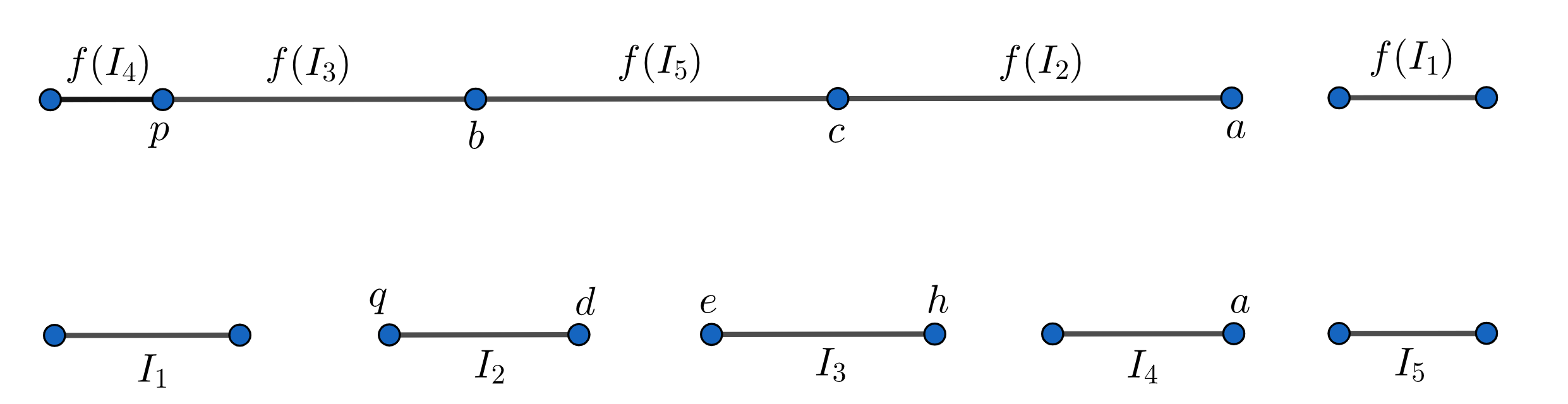}
    \caption{Example 2.7}
    \label{fig:my_label}
\end{figure}
\begin{exmp}
See Figure 3 for the information of the covering system. And recall that we always assume end points of $I_i$ are mapped to end points of $f(I_i)$.\par
Case (1): For $I_2$, assume $f(d)=a,\ f(q)=c$. In this case, we consider the interval $I_5$. We see that $d,e\in f(I_5)$, and $f(d), f(e)\in f^2(I_5)$. Since $f$ is continuous, $f^2(I_5)\supseteq [f(e),f(d)]=[f(e),a]\supseteq [b,a]\supseteq I_3\cup I_4$. Therefore, $f^3(I_5)\supseteq f(I_3\cup I_4)\supseteq I_1$. And finally $f^4(I_5)\supseteq f(I_1)\supseteq I_5$. By example 2.1, we know there is an $x_0\in I_5$ with $f^4(x_0)=x_0$.\par
Case (2): For $I_3$, assume $f(e)=p,\ f(h)=b$. We argue in a similar way to deduce $f^5(I_5)\supseteq I_5$. And again we come back to example 2.1.\par
Case (3): None of the conditions in the above two cases are satisfied. Then $f(d)=c,\ f(e)=b,\ f(q)=a,\ f(h)=p$. This time we consider the point $f(c)$. And after similar discussions, we will obtain: $f^4(I_2)\supseteq I_2,$ if $f(c)\geq q$; $f^5(I_5)\supseteq I_5$, if $f(c)<q.$ Anyway, there is a periodic point of period $\le 5$ in $I_1\cup...\cup I_5$.

In example 2.7, we no longer have the property mentioned in example 2.6. However, we can study the image of some intervals and the behavior of their end points under the mapping $f$ or even $f^2$ to reduce the problem to example 2.1. Note that we frequently use the fact: if $f(I)$ contains $x,y\ (x<y)$ where $I$ is an interval, then $f(I)\supseteq [x,y]$.
\end{exmp}

\begin{exmp}
See Figure 4.\par
Of course we can check in the same way as in example 2.6 and 2.7. However, we can simplify the problem by using the induction trick. Indeed, in this example, we can consider another covering system $(I_1,I_2;f^2)$ instead of the given one to reduce it to the case $k=2$.
\begin{figure}[ht]
    \centering
    \includegraphics[width=\textwidth]{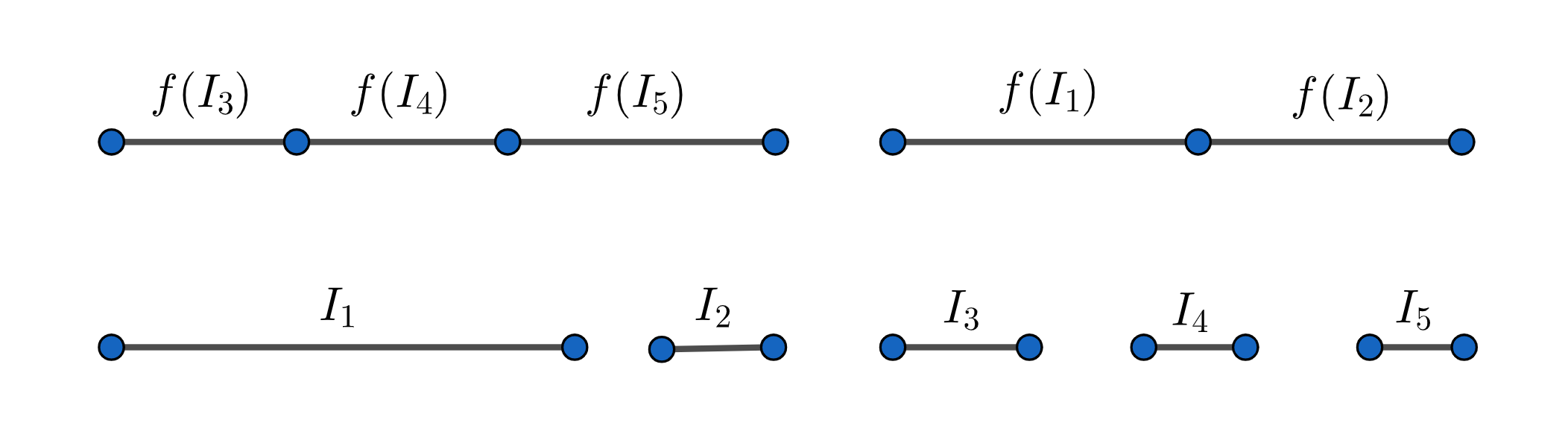}
    \caption{Example 2.8}
    \label{fig:my_label}
\end{figure}
\end{exmp}

\begin{rem}
According to the ordering of all the end points of $I_i$ and $f(I_i)$, we can reduce Conjecture 1.2 to finitely many cases when $k=5$. However, since there are about $(5!)^2$ cases in total, it is impossible to list them all in this paper. We claim here without proof that all the cases can be checked using the methods in the above three examples. Thus Theorem 1.3 holds.  
\end{rem}

\section{Discrete problems}

In this section we show that conjecture 1.2 may be proved by solving another discrete problem related to it. The main idea comes from our attempts to find a general method when trying to prove Conjecture 1.2 for $k=5$. And we now explain it.

As we have seen in the examples in section 2, it will be helpful for simplification if the image of an interval contains entirely another one or more intervals, like in example 2.3. However, not all mappings have this nice property. Therefore, we have to discuss where the end points of each interval go under the mapping $f$ or even $f^2$, just like what we have done in example 2.7. So this inspires us to divide the initial intervals into smaller subintervals, such that the image of each subinterval contains the whole of some other subintervals, not just part of them.

For example, if we cut the intervals in example 2.7 into small pieces, then more details of the mapping $f$ will be exhibited in the graph, besides the behavior of the end points of each interval under $f$. (See Figure 5.) However, it may happen that no matter how you divide the intervals, there will be an interval whose image under $f$ contains only part of another interval. What saves us here is that, we can consider perturbation of $f$ if the intervals are cut into enough small pieces. Indeed, we have the following lemma:

\begin{lem}
If there is a covering system $(I_1,I_2,\ldots,I_k;f)$ such that $$\forall x\in I_1\cup \ldots\cup I_k,\quad and\quad \forall l\le k,\quad f^l(x)\neq x,$$ then $\exists$ $\delta >0$, $s.t.$, whenever $\lVert g-f\rVert<\delta$, it holds that $g^l(x)\neq x$, for $l\le k$, $x\in I_1\cup \ldots \cup I_k$.
\end{lem}
\begin{proof}
Recall that $f$ has a periodic point of period $\le k$ in $I_1\cup \ldots\cup I_k$ is equivalent to that $F(x)=(f(x)-x)\ldots(f^k(x)-x)$ has a zero in $I_1\cup \ldots\cup I_k$. Therefore, the proof is completed since $f$ is uniformly continuous on a compact set containing $f^k(I_1\cup \ldots\cup I_k)$.
\end{proof}

Now we can introduce the discrete problem related to Conjecture 1.2. Consider the covering system in Lemma 3.1. Let $a_1<b_1<a_2<b_2<\ldots<a_k<b_k$ be all the end points of the intervals $I_1,\ldots,I_k$ and denote $M_0=\{a_1,b_1,\ldots,a_k,b_k\}$. Consider their images $f(M_0)=\{ f(a_1),\ldots,f(b_k)\}$ and we eliminate the points which don't belong to $I_1\cup \ldots\cup I_k$ to obtain a set $S_1\subseteq f(M_0)$. Put $M_1=M_0\cup S_1$. Inductively, in each step we consider the image of $M_i$ under the mapping $f$ and we eliminate those points dropping out of $I_1\cup ...\cup I_k$ to get a set $S_{i+1}\subseteq f(M_i)$. Then let $M_{i+1}=M_i\cup S_{i+1}$. 

Finally we obtain a sequence of sets $M_0\subseteq M_1\subseteq \ldots \subseteq M_i\subseteq \ldots$, where each $M_i$ contains the images of the initial end point set $M_0$ under \{$id,f,f^2,\ldots,f^i$\} except for those dropping out of $I_1\cup \ldots\cup I_k$.

Then we can find a sufficiently large integer $N$ such that 
$\forall x\in M_N-M_{N-1}$, $dist(x, M_{N-1})<\delta$, where $\delta$ comes from Lemma 3.1. Therefore, we can take small perturbation of $f$ to get a map $\tilde{f}$ with $\lVert \tilde{f}-f\rVert<\delta$, and $\widetilde{M_{N-1}}=\widetilde{M_N}$ for $\tilde{f}$. Indeed, we just move each point $(x_0,f(x_0))\in M_{N-1}\times M_{N}$ on the graph of $f$ to the nearest point $(x_0,y_0)\in \{x_0\}\times M_{N-1}$ without destroying the continuity. Thus, if we divide the intervals into small pieces using the points in $\widetilde{M_N}$, then $\tilde{f}$ will have the nice property that the image of each subinterval contains entirely another one or more subintervals. In other words, $\tilde{f}$ can be regarded as a discrete mapping from $\{1,2,\ldots,n\}$ to subsets of $\{1,2,\ldots,n\}$. (See Conjecture 1.4.) Here each number $j$ represents a small piece of interval $J_j$. And $j'\in f(j)$ if $f(J_j)\supseteq J_{j'}$. Note that the image of some interval may be empty, since we ignore the part outside $I_1\cup \ldots\cup I_k$. And we only care about the interval determined by $\{f(a),f(b)\}$ if $a,b\in \widetilde{M_{N}}$, although the image $f([a,b])$ may be larger.

It should also be remarked here that the only useful information of $f$ on the so called ``gaps" (see \cite{1988Periodic}) between two adjacent intervals $I_i$ and $I_{i+1}$ is the behavior of the end points of each gap. In other words, for the image $f(J_j)$ of each small subinterval $J_j$, we only focus on the part $f(J_j)\cap (I_1\cup\ldots\cup I_k)$. This is the reason why we eliminate the points outside $I_1\cup\ldots\cup I_k$ in each step. We believe this information is enough for the proof because in thousands of examples for $k\le 5$ it does suffice to get a conclusion.

To sum up, once we have a mapping $f$ satisfying the conditions in Lemma 3.1, we can find an $\tilde{f}$ close enough to $f$, which also satisfies the conditions but can be viewed as a discrete mapping. Such $f$ and $\tilde{f}$ are counterexamples to Conjecture 1.2. Therefore, if we argue by contradiction, we can prove Theorem 1.5, which says that Conjecture 1.2. is a corollary of Conjecture 1.4. 

\begin{proof}[Proof of Theorem 1.5]
Suppose the covering system $(I_1,I_2,\ldots,I_k;f)$ is a counterexample to Conjecture 1.2. Then as mentioned above, after perturbation, we can assume $f$ satisfies the conditions of Conjecture 1.4. Therefore the conclusion of Conjecture 1.4 gives us a subinterval of $I_j$ represented by $conv\{r,s\}$ whose image under $f^l$ contains itself. Thus there must exist a periodic point of period $\le l\le k$ in $I_j$, a contradiction.  
\end{proof}

Although Conjecture 1.4 reduces Conjecture 1.2 to a discrete problem, it is not satisfactory since it involves something inconvenient to deal with, like convex hull and partition. In the following we will simplify the conditions in Conjecture 1.4 and find another description of it which only depends on the own property of a permutation.

Let $f$ be in Conjecture 1.4. We can assume further:\par
(1)$\quad \forall i\neq j$, $f(i)\cap f(j)= \emptyset$.\par
(2)$\quad \forall i$, $f(i) \neq \emptyset$.\par
(3)$\quad \forall i$, $f(i)$ contains exactly one element.\par
(4)$\quad f$ is a transitive permutation in the symmetry group $S_n$, $i.e.$, $f$ can be written as $(i_1i_2\ldots i_n)$.\par
Indeed, (1) can be realized since we can move out the common part of $f(i)$ and $f(j)$ from one of them without changing the property $\bigcup_{i=1}^n f(i)=\{1,2,\ldots,n\}$.\par
(2) is also satisfied because we can restrict the mapping $f$ to $\{1,2,\ldots,n\}-\{i\}$ if necessary.\par
(3) is a consequence of (1), (2) and the condition $\bigcup_{i=1}^n f(i)=\{1,2,\ldots,n\}$.\par
For (4), note that $f \in S_n$ because of (3). Suppose $f$ is not transitive. Then choose an orbit $(i_1\ldots i_m)$ of $f$ ($m<n$) and restrict both $f$ and the partition $I_1,\ldots,I_k$ to $\{i_1,\ldots,i_m\}$.

In conclusion, we obtain:
\begin{prop}
It is sufficient to consider $f\in S_n$ and $f$ is transitive in Conjecture 1.4.
\end{prop}

\begin{figure}[ht]
    \centering
    \includegraphics[width=\textwidth]{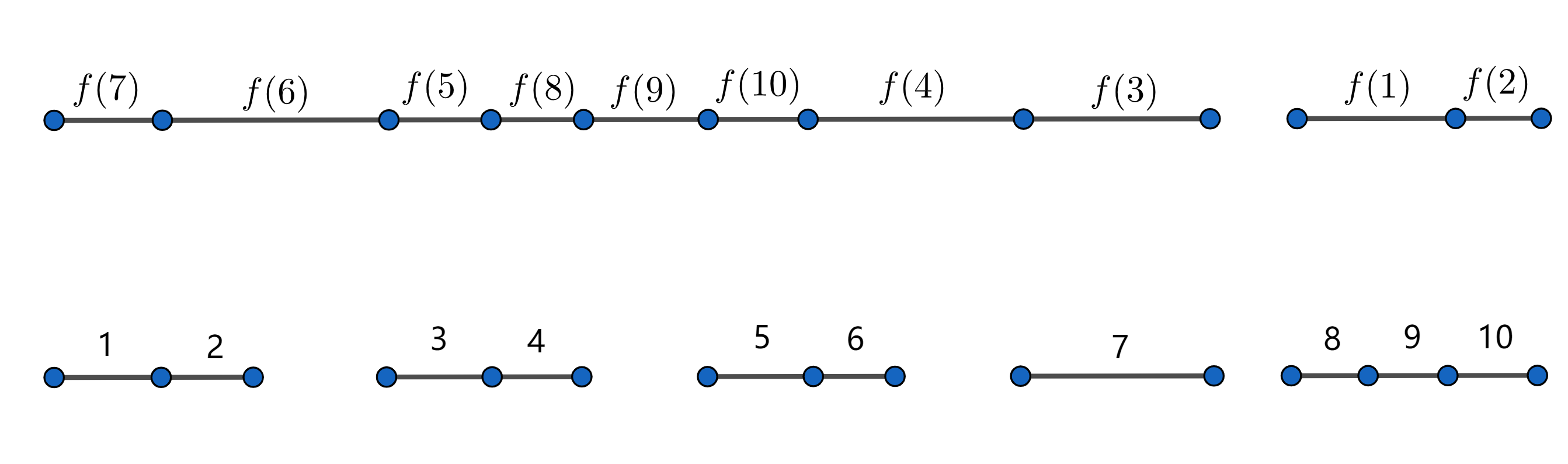}
    \caption{Example 3.3}
    \label{fig:my_label}
\end{figure}

\begin{exmp}
We explain in this example how to obtain a discrete mapping in Conjecture 1.4.\par
For the covering system in Example 2.7 (Figure 3), we cut the initial intervals into ten pieces:
$$I_1=\{1,2\},\ I_2=\{3,4\},\ \ldots,\ I_5=\{8,9,10\}.$$
And the image of each piece under $f$ happens to be some other pieces. (See Figure 5.) Therefore $f$ can be regarded as a discrete mapping as follows:
$$f(1)=\{8,9\},\ f(2)=\{10\},\ f(3)=7,\ f(4)=6,\ \ldots,\ f(9)=\emptyset,\  f(10)=\{5\}.$$
Moreover, $f$ can be reduced to a transitive permutation in $S_9$ as $(1\to8\to4\to6\to2\to10\to5\to3\to7\to1)$. And $I_5$ reduces to $\{8,10\}$, respectively. Note here that if we follow the discussion at the beginning of this section, there will be no need for a point between 8 and 9, since it is not the image of the initial end points under any $f^i$. We add this point in order to show how we reduce the mapping $f$ to a transitive permutation in general cases.

\end{exmp}

Now we turn to the proof of the equivalence of Conjecture 1.7 and Conjecture 1.4. (See also Definition 1.6.)

\begin{prop}
Conjecture 1.7 implies Conjecture 1.4.
\end{prop}
\begin{proof}
As mentioned above, we can assume $f\in S_n$ is transitive, without loss of generality. \par
Suppose Conjecture 1.7 is true. Then for any partition 
$$I_1=\{1,2,\ldots,i_1\},\quad I_2=\{i_1+1,\ldots,i_2\},\quad \ldots,\quad I_k=\{i_{k-1}+1,...,n\},$$
the $(k-1)$ characteristic numbers $\{m_{i_1},m_{i_2},\ldots,m_{i_{k-1}}\}$ cannot cover $\{m_1',\ldots,m_k'\}$. Therefore, $\exists t\neq i_1,\ldots,i_{k-1}$, such that $m_t\le k$, since $m_1'\le \ldots\le m_k'\le k$ in the characteristic sequence. In other words, 
$$(convf)^{m_t}(A_t)\supseteq A_t,\quad where\quad A_t=\{t,t+1\}\subseteq some\ I_j.$$
Therefore, we can just choose $r=t,\ s=t+1$ in Conjecture 1.4.\par
\end{proof}

Recall that in one dimensional dynamic systems we often consider a directed graph associated to a periodic orbit (For more details, see \cite{Uhland1982Interval} or chapter 1 of \cite{onedimdynamics1992}). Namely, if $f$ is a transitive permutation in Definition 1.6, we can construct a directed graph $\Gamma_f$ with $(n-1)$ vertices $\{A_1,A_2,\ldots,A_{n-1}\}$, where $A_i$ is defined in 1.6. And there is an edge from $A_i$ to $A_j$ if and only if $convf(A_i)\supseteq A_j$. Thus, we obtain:

\begin{prop}
For a transitive $f\in S_n$ and its associated directed graph $\Gamma_f$, the characteristic number $m_i$ is equal to the minimal length of cycles starting from $A_i$ in $\Gamma_f$.
\end{prop}

Using directed graphs, we can complete the proof of Theorem 1.8.

\begin{proof}[Proof of Theorem 1.8]
Suppose Conjecture 1.4 is true. For a transitive $f\in S_n$,  suppose on the contrary that there is a $k\le n-1$ such that $m_k'\geq k+1$, and $k$ is the minimal number with this property. Then there are only $(k-1)$ characteristic numbers $m_{i_1},m_{i_2},\ldots,m_{i_{k-1}}\le k$. Without loss of generality, assume $i_1<\ldots<i_{k-1}$. And consider the partition:
$$I_1=\{1,2,\ldots,i_1\},\quad I_2=\{i_1+1,\ldots,i_2\},\quad \ldots,\quad I_k=\{i_{k-1}+1,\ldots,n\}.$$
By Conjecture 1.4, we can find $j,l\le k$ and $r,s\in I_j$ such that
$$(convf)^l(\{r,s\})\supseteq \{r,s\}.$$
Now if we extend $f$ to be a piecewise linear mapping from the interval $[1,n]$ to $[1,n]$, then we will obtain $f^l([r,s])\supseteq [r,s]$. Therefore, there exists $x_0\in [r,s]$ with $f^l(x_0)=x_0$. Note that since $f$ is transitive as a permutation in $S_n$, $x_0$ cannot belong to $\{1,2,\ldots,n\}$. Consequently, $x_0\in (t,t+1)$ for some integer $t\in [r,s]$  and it is similar for $f(x_0),f^2(x_0),\ldots,f^{l-1}(x_0)$. Since $f$ is piecewise linear, $x_0\in (t,t+1)$, $f(x_0)\in (t',t'+1)$ imply $f([t,t+1])\supseteq [t',t'+1]$, and it is similar for other pairs of points. So the orbit $\{x_0,f(x_0),\ldots,f^{l-1}(x_0),f^l(x_0)\}$ gives a cycle of length $l\le k$ starting from the vertex $A_t$ in the associated directed graph $\Gamma_f$, which implies $m_t\le k$. But this $t$ together with $i_1,\ldots,i_{k-1}$ gives $k$ characteristic numbers $\le k$, contradicting $m_{k}'\geq k+1$.\par
Combined with Proposition 3.4, the proof is completed.
\end{proof}

\begin{rem}
(1) The condition that $f$ is transitive cannot be removed in Conjecture 1.7. Because if $f\in S_3$ and $f(1)=3,\ f(2)=2,\ f(3)=1$, then the characteristic numbers will be $m_1=m_2=2$.\par
(2) $\{m\arrowvert\ (convf)^m(A_i)\supseteq A_i\}$ cannot be replaced by $\{m\arrowvert (f^m(i)-i)(f^m(i+1)-(i+1))<0\}$ in the definition of the characteristic number in Conjecture 1.7. As an example, consider $f$: $(1\to 2\to 4\to 3\to 1)$. Then the characteristic numbers will become $m_1=m_3=3,\ m_2=1$, if we don't take convex hull in each step.  
\end{rem}

We summarize what we have done. We show in Theorem 1.5 that the initial problem on existence of periodic points in a covering system follows from a discrete one. And we only need to focus on the transitive permutations by Proposition 3.2. Finally, Conjecture 1.7 and Theorem 1.8 tell us that it is sufficient to study the characteristic sequence of a transitive permutation, which is a property depending only on the permutation itself.

At the end of this section, we give some examples of periodic orbits whose characteristic sequences are known.
\begin{exmp}
(1) Let $f\in S_n$ be $(1\to 2\to \ldots\to n\to 1)$. Then the characteristic sequence of $f$ is $1\le 2\le \ldots\le n-1$.\par
(2) Let $f\in S_{2n+1}$ be of Stefan type (see \cite{P1977A} for more details), that is, 
$$1\to (n+1)\to (n+2)\to n\to (n+3)\to (n-1)\to\ldots\to 2n \to 2\to (2n+1)\to1.$$
Then the characteristic sequence of $f$ is $$1\le 2\le 2\le 4\le 4\le\ldots\le 2n-2\le 2n-2\le 2n.$$\par
(3) It has been checked by computer program that Conjecture 1.7 holds for all $n\le 12$. And in \cite{10.2307/2690145}, the author gives all the directed graph associated to transitive permutations in $S_5$.
\end{exmp}

\section*{Acknowledgement}
I would like to thank Prof. S.A. Bogatyi for introducing the problem to me. And I thank my supervisor, Prof. Tian Gang, for his encouragement as well as helpful suggestions. I'm also grateful to Dr. Ye Yanan and Dr. Zhao Yikai for writing computer programs to verify some claims.

\bibliographystyle{unsrt}
\bibliography{ref}

\end{document}